\documentclass[preprint]{elsarticle}
\usepackage{amsfonts}
\usepackage{amsmath}
\usepackage{amssymb}

\newtheorem{theorem}{Theorem}

\newtheorem{corollary}[theorem]{Corollary}

\newtheorem{definition}[theorem]{Definition}

\newtheorem{lemma}[theorem]{Lemma}

\newtheorem{proposition}[theorem]{Proposition}
\newtheorem{remark}[theorem]{Remark}

\newenvironment{proof}[1][Proof]{\noindent \textbf{#1.} }{\  \rule{0.5em}{0.5em}}

\begin{document}

\begin{frontmatter}
	
	\title{Asymptotic properties for a general class of Sz\'asz-Mirakjan-Durrmeyer operators
	}

	\author[1]{Ulrich Abel} 
	\author[2]{Ana Maria Acu}
	\author[3]{Margareta Heilmann}
	\author[4]{Ioan Ra\c sa}

	\vspace{10cm}
	

	\address[1] {Fachbereich Mathematik, Naturwissenschaften und Datenverarbeitung,
		Technische Hochschule Mittelhessen, Germany, 
		e-mail: ulrich.abel@mnd.thm.de }
	\address[2]{Department of Mathematics and Informatics, Lucian Blaga University of Sibiu,  Romania ({\it Corresponding Author}), e-mail: anamaria.acu@ulbsibiu.ro}
	\address[3]{School of Mathematics and Natural Sciences, University of Wuppertal,  Germany, e-mail: heilmann@math.uni-wuppertal.de}
	\address[4]{Technical University of Cluj-Napoca, Faculty of Automation and Computer Science, Department of Mathematics, Str. Memorandumului nr. 28, 400114 Cluj-Napoca, Romania
		e-mail:  ioan.rasa@math.utcluj.ro }

	\begin{abstract} 	
		{
		In this paper we introduce a general family of Szz\'asz--Mirakjan--Durrmeyer type operators depending on an integer parameter $j \in \mathbb{Z}$. 
		They can be viewed as a generalization of the Sz\'asz--Mirakjan--Durrmeyer operators \cite{MaTo1985}, Phillips operators \cite{Ph1954} and corresponding Kantorovich modifications of higher order. 

For $j\in {\mathbb{N}}$, these operators possess the exceptional property to preserve constants and the monomial $x^{j}$. It turns out, that an extension of this family covers certain well-known operators studied before, so that the outcoming results could be unified.

We present the complete asymptotic expansion for the sequence of these operators. All its coefficients are given in a concise form. In order to prove the expansions for the class of locally integrable functions of exponential growth on the positive half-axis, we derive a localization result which is interesting in itself. 
		} 
	\end{abstract}
	
	
\end{frontmatter}

\section{Introduction}

For fixed $j \in \mathbb{Z}$ we consider sequences of positive linear
operators $S_{n,j}$ which can be viewed as a generalization of the
Sz\'asz-Mirakjan-Durrmeyer operators \cite{MaTo1985}, Phillips operators 
\cite{Ph1954} and corresponding Kantorovich modifications of higher order.

The operators $S_{n,j}$, $j\in {\mathbb{Z}}$, approximate locally integrable
functions of exponential growth on the positive half-axis in each point of
continuity. For $j\in {\mathbb{N}}$, these operators possess the exceptional
property to preserve constants and the monomial $x^{j}$.

The idea of this paper was to construct a family of operators of
Sz\'asz--Mirakjan--Durrmeyer type preserving constants and $x^j$, $j\in {%
\mathbb{N}}$, similar to the well-known Phillips operators for $j=1$.

It turns out, that an extension of this family covers certain well-known
operators studied before, so that the outcoming results could be unified. We
will investigate pointwise approximation for continuous functions satisfying
an exponential growth condition. Furthermore, we derive asymptotic
expansions for the sequence $\left( \left( S_{n,j}f\right) \left( x\right)
\right) _{n\geq 1}$ even for simultaneous approximation. All its
coefficients are given in a concise form. In order to prove the expansions
for the class of locally integrable functions of exponential growth on the
positive half-axis, we prove a localization result which is interesting in itself.

For $A\geq 0$, we denote by $E_{A}$ the space of functions $f:[0,\infty
)\longrightarrow \mathbb{R}$, $f$ locally integrable, satisfying the growth
condition $\left\vert f\left( t\right) \right\vert \leq Ke^{At}$, $t\geq 0$
for some positive constant $K$. Furthermore, define $E:=\displaystyle%
\bigcup_{A\geq 0}E_{A}$.

Throughout the paper we use the convention that a sum is zero if the upper
limit is smaller than the lower one.

\begin{definition}
Let $f\in E_{A}$, $n>A$ and $j\in {\mathbb{Z}}$. Then the operators $S_{n,j}$
are defined by 
\begin{equation*}
\left( S_{n,j}f\right) \left( x\right)
=f(0)\sum_{k=0}^{j-1}s_{n,k}(x)+\sum_{k=j}^{\infty
}s_{n,k}(x)n\int_{0}^{\infty }s_{n,k-j}(t)f(t)dt
\end{equation*}%
where 
\begin{equation*}
s_{n,k}(x)=\dfrac{(nx)^{k}}{k!}e^{-nx},\,k\geq 0.
\end{equation*}%
For the sake of simplicity we define $s_{n,k}(x)=0$, if $k<0$.
\end{definition}

For $j\leq 0$ we have 
\begin{equation*}
\left( S_{n,j}f\right) \left( x\right) =\sum_{k=0}^{\infty
}s_{n,k}(x)n\int_{0}^{\infty }s_{n,k-j}(t)f(t)dt.
\end{equation*}%
For the special case $j=0$ we get the Sz\'{a}sz-Mirakjan-Durrmeyer operators
first defined in \cite{MaTo1985}, for $j=1$ the Phillips operators, also
called genuine Sz\'{a}sz-Mirakjan-Durrmeyer operators, \cite{Ph1954}. For $%
j\leq -1$ the operators coincide with certain auxiliary operators, see,
e.g., \cite[(1.4)]{HeiMue1989} for $c=0$, named $M_{n,r}$ there, $r=-j$ or 
\cite[(3.5)]{He1992} for $c=0$, named $M_{n,s}$ there, $s=-j$. They can also
be considered as corresponding Kantorovich modifications of higher order. 

For $\ell =0,1,\ldots $ we denote $e_{\ell }(t)=t^{\ell }$, $t\geq 0$.

All these operators preserve $e_{0}$ and for $j\geq 1$ also $e_{j}$ since 
\begin{eqnarray*}
\left( S_{n,j}e_{j}\right) \left( x\right)  &=&\sum_{k=j}^{\infty }s_{n,k}(x)%
{n\int_{0}^{\infty }s_{n,k-j}(t)t^{j}dt} \\
&=&x^{j}\sum_{k=j}^{\infty }s_{n,k-j}(x)=x^{j}.
\end{eqnarray*}%
Furthermore, for $j\geq 1$, it is easy to see that $\left( S_{n,j}f\right)
\left( 0\right) =f(0)$. 

In the case $j\leq 0$ the operators $S_{n,j}$ approximate at $x=0$, if the
function $f\in E_{A}$ is continuous (from the right) at $x=0$. This can be
seen as follows. Given $\varepsilon >0$, let $\left\vert {f\left( t\right)
-f\left( 0\right) }\right\vert <\varepsilon $, for $0\leq t<\delta $. Define 
$\tilde{f}\left( t\right) =0$, for $0\leq t<\delta $, and $\tilde{f}\left(
t\right) =f\left( t\right) -f\left( 0\right) $, for $t\geq \delta $.\ Then 
\begin{eqnarray*}
\left\vert \left( S_{n,j}f\right) \left( 0\right) -f\left( 0\right)
\right\vert  &\leq &n\int_{0}^{\infty }s_{n,-j}(t)\left\vert f\left(
t\right) -f\left( 0\right) \right\vert dt \\
&< &n\varepsilon \int_{0}^{\delta }s_{n,-j}(t) dt + n\int_{\delta
}^{\infty }s_{n,-j}(t)\left\vert \tilde{f}\left( t\right) \right\vert dt \\
&< &\varepsilon +\left( {S_{n,j}}\vert \tilde{f}\vert \right)
\left( 0\right) .
\end{eqnarray*}%
Since $\vert \tilde{f}\vert \in E$ the localization theorem
(Theorem~\ref{theorem-localization}) implies that ${\left( {S_{n,j}}%
\vert \tilde{f}\vert \right) \left( 0\right) \rightarrow 0}$ as $%
n\rightarrow \infty $.

The following lemma shows that $S_{n,j}$ is a mapping from $E$ to $E$. For $%
A\geq 0$, let $\exp _{A}$ denote the exponential function $\exp _{A}\left(
t\right) =e^{At}$.

\begin{lemma}
\label{lem.1} Let $n>2A$, $A\geq 0$, $f\in E_A$. Then $S_{n,j}f\in E_{2A}$.
\end{lemma}

\begin{proof}
Note that $n> 2A$ implies $\dfrac{n}{n-A}< 2$ and $\dfrac{n-A}{n}\leq1$.

Let $j\leq 0$. Therefore, 
\begin{eqnarray}
\left( S_{n,j}\exp _{A}\right) \left( x\right) &=&\sum_{k=0}^{\infty
}s_{n,k}(x)n\int_{0}^{\infty }s_{n,k-j}(t)e^{At}dt  \notag \\
&=&\sum_{k=0}^{\infty }s_{n,k}(x)\left( \frac{n}{n-A}\right)
^{k-j}n\int_{0}^{\infty }s_{n-A,k-j}(t)dt  \notag \\
&=&\left( \frac{n}{n-A}\right) ^{-j+1}\sum_{k=0}^{\infty }s_{n,k}(x)\left( 
\frac{n}{n-A}\right) ^{k}  \notag \\
&=&\left( \frac{n}{n-A}\right) ^{-j+1}e^{Ax\frac{n}{n-A}}\sum_{k=0}^{\infty
}s_{\frac{n^{2}}{n-A},k}(x)  \notag \\
&=&\left( \frac{n}{n-A}\right) ^{-j+1}e^{Ax\frac{n}{n-A}}\leq
2^{-j+1}e^{2Ax}.  \notag
\end{eqnarray}

For $j\geq 1$ we use $\dfrac{n-A}{n}\leq 1$. Thus 
\begin{align}
\left( S_{n,j}\exp _{A}\right) \left( x\right) & =\displaystyle%
\sum_{k=0}^{j-1}s_{nk}(x)+n\displaystyle\sum_{k=j}^{\infty }s_{nk}(x)\left( 
\dfrac{n}{n-A}\right) ^{k-j}\int_{0}^{\infty }s_{n-A,k-j}(t)dt  \notag \\
& \leq 1+\left( \dfrac{n-A}{n}\right) ^{j-1}e^{Ax\frac{n}{n-A}%
}\sum_{k=j}^{\infty }s_{\frac{n^{2}}{n-A},k}(x)  \notag \\
& \leq 1+e^{2Ax}.  \notag
\end{align}%
This completes the proof.
\end{proof}

\begin{remark}
The proof reveals that 
\begin{equation*}
\left( S_{n,j}\exp _{A}\right) \left( x\right) \leq 1+\left( \dfrac{n-A}{n}%
\right) ^{j-1}e^{\frac{An}{n-A}x}\text{ \qquad }\left( n>A\right) ,
\end{equation*}%
for each integer $j$.
\end{remark}


In order to facilitate the formulation of our results we use a suited
differential operator $\mathcal{D}_{j}^{2}$. For $j\in {\mathbb{Z}}$ define 
\begin{equation}
\mathcal{D}_{j}^{2}=\left( 1-j\right) D+e_{1}D^{2},  \label{eq.X2}
\end{equation}%
where $D$ denotes the ordinary differentiation operator. We mention that $%
\mathcal{D}_{j}^{2}e_{0}=\mathcal{D}_{j}^{2}e_{j}=0$. More generally, we
define 
\begin{equation}
\mathcal{D}_{j}^{2k}=e_{j}D^{k}\left( e_{k-j}D^{k}\right) ,\,k\in {\mathbb{N}%
}.  \label{eq.X3}
\end{equation}%
For $k=1$ this is in agreement with (\ref{eq.X2}). Moreover, define $%
D_{j}^{0}$ to be the identity operator. It can be shown that $\mathcal{D}%
_{j}^{2k}$ can be written as $k$-th iterate of $\mathcal{D}_{j}^{2}$, i.e., 
\begin{equation}
\mathcal{D}_{j}^{2k}=(\mathcal{D}_{j}^{2})^{k}.  \label{eq.X1}
\end{equation}

\section{Moments and central moments of the operators}

For the asymptotic relations in Section~\ref{section-Asymptotic expansions}
we will need central moments. First we consider the images of monomials. For 
$j\leq 0$ we refer to \cite[Satz 4.2, (4.4)]{He1992}.

\begin{proposition}
\label{prop-moments}The moments of the operators $S_{n,j}$ are given by 
\begin{equation*}
S_{n,j}e_{0}=e_{0},
\end{equation*}
and for $r\in \mathbb{N}$, by 
\begin{equation}  \label{Snjj-moments-general}
\left( S_{n,j}e_{r}\right) \left( x\right) = \sum_{k=0}^{r}\dfrac{1}{n^k}{%
\binom{r}{k}}(r-j)^{\underline{k}}x^{r-k} -n^{-r}e^{-nx}\sum_{k=0}^{j-1-r}%
\frac{\left( nx\right) ^{k}}{k!}\left( k+r-j\right) ^{\underline{r}}.
\end{equation}
\end{proposition}

\begin{remark}
\begin{itemize}
\item[1)] The second sum in (\ref{Snjj-moments-general}) disappears if $%
r\geq j$. Consequently, it holds 
\begin{equation}
\left( S_{n,j}e_{r}\right) \left( x\right) =\sum_{k=0}^{r}\frac{1}{n^{k}}{%
\binom{r}{k}}(r-j)^{\underline{k}}x^{r-k}\text{ \qquad }\left(
r=j,j+1,\ldots \right) .  \label{Snjj-moments-exakt}
\end{equation}

\item[2)] Using the definition of $\mathcal{D}_{j}^{2k}$ in (\ref{eq.X3}) we
can rewrite for each $j$ the moments as 
\begin{equation*}
\left( S_{n,j}e_{r}\right) \left( x\right) =\displaystyle\sum_{k=0}^{r}%
\dfrac{1}{n^{k}}\left( \mathcal{D}_{j}^{2k}e_{r}\right)
(x)-n^{-r}e^{-nx}\sum_{k=0}^{j-1-r}\dfrac{(nx)^{k}}{k!}(k+r-j)^{\underline{r}%
}.
\end{equation*}
\end{itemize}
\end{remark}

\begin{remark}
The case $r=j$ in (\ref{Snjj-moments-exakt}) shows again the preservation of
the monomial $e_{j}$.
\end{remark}

\begin{remark}
The moments of the operators $S_{n,j}$ satisfy, for $r=0,1,2,\ldots $, the
asymptotic relation 
\begin{equation}
\left( S_{n,j}e_{r}\right) \left( x\right) = \sum_{k=0}^{r}\frac{1}{n^{k}k!}
({\mathcal{D}}^{2k}_j e_r) (x) +O\left( n^{j-1-2r}e^{-nx}\right) \text{
\qquad }\left( n\rightarrow \infty \right),  \label{Snjj-moments-asymptotic}
\end{equation}
which appears to be an identity without Landau-O in the case $r\geq j$.
\end{remark}

\begin{proof}[Proof of Proposition~\protect\ref{prop-moments}]
We already know that $S_{n,j}e_{0}=e_{0}$. For $r\in \mathbb{N}$, we have 
\begin{eqnarray*}
\left( S_{n,j}e_{r}\right) \left( x\right) &=& \sum_{k=j}^{\infty}s_{n,k}(x)
n\int_{0}^{\infty }s_{n,k-j}\left( t\right)t^{r}dt \\
&=& n^{-r}e^{-nx}\sum_{k=j}^{\infty }\frac{\left( nx\right) ^{k}}{k!}
\left(k-j+r\right)^{\underline{r}}=n^{-r}e^{-nx}\left( T_{1}+T_{2}\right) ,
\end{eqnarray*}
say, where 
\begin{equation*}
T_{2}:=\sum_{k=0}^{j-1}\frac{\left( nx\right) ^{k}}{k!}\left( k-j+r\right)
^{ \underline{r}}=\sum_{k=0}^{j-1-r}\frac{\left( nx\right) ^{k}}{k!}\left(
k-j+r\right) ^{\underline{r}},
\end{equation*}
since $\left( k-j+r\right) ^{\underline{r}}=0$ if $j-r\leq k\leq j-1$, and 
\begin{equation*}
T_{1}:=\sum_{k=0}^{\infty }\frac{\left( nx\right) ^{k}}{k!}\left(
k-j+r\right) ^{\underline{r}}.
\end{equation*}
By Vandermonde convolution, 
\begin{equation*}
\left( k-j+r\right) ^{\underline{r}}=\sum_{i=0}^{r}\binom{r}{i}k^{\underline{
i}}\left( r-j\right) ^{\underline{r-i}}
\end{equation*}
and we obtain 
\begin{equation*}
T_{1}=\sum_{i=0}^{r}\binom{r}{i}\left( r-j\right) ^{\underline{r-i}
}\sum_{k=i}^{\infty }\frac{\left( nx\right) ^{k}}{k!}k^{\underline{i}
}=e^{nx}\sum_{i=0}^{r}\binom{r}{i}\left( r-j\right) ^{\underline{r-i}}\left(
nx\right) ^{i}.
\end{equation*}
\end{proof}

Now we consider the central moments. For $x\in \mathbb{R}$ we use the
notation $\psi _{x}(t)=t-x$.

\begin{proposition}
\label{prop-central-moments}The central moments of the operators $S_{n,j}$
are given by 
\begin{equation*}
S_{n,j}\psi _{x}^{0}=e_{0},
\end{equation*}
and for $s\in \mathbb{N}$, by 
\begin{eqnarray*}
\left( S_{n,j}\psi _{x}^{s}\right) \left( x\right)
&=&\sum_{k=0}^{\left\lfloor s/2\right\rfloor }\frac{x^{k}}{k!n^{s-k}}s^{%
\underline{2k}}\left( s-k-j\right) ^{\underline{s-2k}} \\
&&-e^{-nx}\sum_{r=1}^{s}\binom{s}{r}\left( -x\right)
^{s-r}n^{-r}\sum_{k=0}^{j-1-r}\frac{\left( nx\right) ^{k}}{k!}\left(
k+r-j\right) ^{\underline{r}}.
\end{eqnarray*}
\end{proposition}

\begin{proof}
Obviously, we have $S_{n,j}\psi _{x}^{0}=S_{n,j}e_{0}=e_{0}$. For $s\in 
\mathbb{N}$, we have 
\begin{eqnarray*}
\left( S_{n,j}\psi _{x}^{s}\right) \left( x\right) &=&\sum_{r=0}^{s}\binom{s%
}{r}\left( -x\right) ^{s-r}\left( S_{n,j}e_{r}\right) \left( x\right) \\
&=&\sum_{r=0}^{s}\binom{s}{r}\left( -x\right) ^{s-r}\sum_{k=0}^{r}\binom{r}{k%
}\left( r-j\right) ^{\underline{k}}n^{-k}x^{r-k} \\
&&-e^{-nx}\sum_{r=1}^{s}\binom{s}{r}\left( -x\right)
^{s-r}n^{-r}\sum_{k=0}^{j-1-r}\frac{\left( nx\right) ^{k}}{k!}\left(
k+r-j\right) ^{\underline{r}} \\
=: &&M_{n,j,s,1}(x)+M_{n,j,s,2}(x).
\end{eqnarray*}%
In \cite[Korollar 4.4 (4.6)]{He1992} it is proved for $j\leq 0$ that 
\begin{equation*}
M_{n,j,s,1}(x)=\sum_{k=0}^{\left\lfloor s/2\right\rfloor }\frac{x^{k}}{%
n^{s-k}}\frac{s!}{k!(s-2k)!}\left( s-k-j\right) ^{\underline{s-2k}}.
\end{equation*}%
The proof for $j\geq 1$ is the same.
\end{proof}

\begin{remark}
Note that $\left( s-k-j\right) ^{\underline{s-2k}}=0$ if $0\leq s-k-j\leq
s-2k$, i.e., $k\leq s-j$ and $k\leq j$, so $k\leq \min \left\{ j,s-j\right\} 
$. \newline
If $j\leq s\leq 2j-1$ we have $\min \left\{ j,s-j\right\} =s-j$. Therefore,
we get 
\begin{equation*}
M_{n,j,s,1}(x)=\sum_{k=s-j+1}^{\left\lfloor s/2\right\rfloor }\frac{x^{k}}{%
k!n^{s-k}}s^{\underline{2k}}\left( s-k-j\right) ^{\underline{s-2k}}.
\end{equation*}%
In particular, if $s=2j-1$, we have 
\begin{equation*}
M_{n,j,s,1}(x)=\sum_{k=j}^{j-1}\frac{x^{k}}{k!n^{s-k}}s^{\underline{2k}%
}\left( s-k-j\right) ^{\underline{s-2k}}=0.
\end{equation*}%
Furthermore, if $s\geq 2j$, 
\begin{equation*}
M_{n,j,s,1}(x)=\sum_{k=j}^{\left\lfloor s/2\right\rfloor }\frac{x^{k}}{%
k!n^{s-k}}s^{\underline{2k}}\left( s-k-j\right) ^{\underline{s-2k}}.
\end{equation*}
\end{remark}

\begin{remark}

Obviously, $M_{n,j,s,2}\left( x\right) =0$ if $j\leq 1$. If $j\geq 2$,
observe that $\left( -1\right) ^{r}\left( k+r-j\right) ^{\underline{r}%
}=\left( -k-r+j\right) ^{\overline{r}}=\left( j-1-k\right) ^{\underline{r}%
}>0 $, for $k=0,\ldots ,j-1-r$. Therefore, we infer that 
\begin{equation*}
-\left( -1\right) ^{s}M_{n,j,s,2}\left( x\right) =e^{-nx}\sum_{r=1}^{s}%
\binom{s}{r}x^{s-r}n^{-r}\sum_{k=0}^{j-1-r}\frac{\left( nx\right) ^{k}}{k!}%
\left( j-1-k\right) ^{\underline{r}}\geq 0.
\end{equation*}%
We have the easy estimate 
\begin{equation*}
\left\vert M_{n,j,s,2}\left( x\right) \right\vert \leq \frac{K_{s,j}}{n^{s}},
\end{equation*}%
where%
\begin{equation*}
K_{s,j}:=\sup_{t\geq 0}e^{-t}\sum_{r=1}^{s}\binom{s}{r}t^{s-r}%
\sum_{k=0}^{j-1-r}\frac{\left( nx\right) ^{k}}{k!}\left( j-1-k\right) ^{%
\underline{r}}.
\end{equation*}%
Obviously, $K_{s,j}=0$ if $j\leq 1$.
\end{remark}

\begin{remark}
For $M_{n,j,s,2}$ we have the representation 
\begin{eqnarray*}
&&M_{n,j,s,2}\left( x\right) =-e^{-nx}\sum_{r=1}^{s}\binom{s}{r}\left(
-x\right) ^{s-r}n^{-r}\sum_{k=0}^{j-1-r}\frac{\left( nx\right) ^{k}}{k!}
\left( k+r-j\right) ^{\underline{r}} \\
&=&-e^{-nx}\left( -x\right) ^{s}\sum_{k=0}^{j-1}\frac{\left( nx\right) ^{k} 
}{k!}\sum_{r=1}^{s}\frac{1}{r!}\left( -s\right) ^{\overline{r}}\left(
k-j+1\right) ^{\overline{r}}\frac{1}{\left( nx\right) ^{r}} \\
&=&-e^{-nx}\left( -x\right) ^{s}\sum_{k=0}^{j-1}\frac{\left( nx\right) ^{k} 
}{k!}\left[ \text{ }_{2}F_{0}\left( k-j+1,-s;;\left( nx\right) ^{-1}\right)
-1\right].
\end{eqnarray*}
Using \cite[(13.1.10)]{AS} we get 
\begin{equation*}
_{2}F_{0}\left( k-j+1,-s;;\left( nx\right) ^{-1}\right) =\left( -nx\right)
^{k-j+1}U\left( k-j+1,k-j+2+s,-nx\right) ,
\end{equation*}
where $U\left( a,b,x\right) $ denotes the confluent hypergeometric function
which is the logarithmic solution of Kummer's equation.
\end{remark}

\section{Rate of convergence of the operators $S_{n,j}$}

Let $C_{b}\left[ 0,+\infty \right) $ be the space of all real-valued,
continuous and bounded functions defined on the interval $\left[ 0,+\infty
\right) $. The following theorem gives an estimate of the rate of
convergence of the operators $S_{n,j}$ in terms of the ordinary modulus of
continuity.

\begin{theorem}
\label{theorem-rate of convergence}Let $j\in \mathbb{Z}$ and $x>0$. For $%
f\in C_{b}\left[ 0,+\infty \right) $, the rate of convergence of the
operators $S_{n,j}$ can be estimated by 
\begin{equation}
\left\vert \left( S_{n,j}f\right) \left( x\right) -f\left( x\right)
\right\vert \leq \left( 1+\sqrt{2x+\frac{\left( j-1\right) \left( j-2\right) 
}{n}}\right) \omega \left( f,\frac{1}{\sqrt{n}}\right) .
\label{estimate-rate of convergence}
\end{equation}
\end{theorem}

We remark that Theorem~\ref{theorem-rate of convergence} implies that $%
S_{n,j}f$ approximates each continuous function $f\in E$. Given $x>0$, we
can choose a function $\widetilde{f}\in C_{b}\left[ 0,+\infty \right) $
coinciding with $f$ on a finite interval $\left[ 0,R\right) $ with $R>x$. By
the localization theorem (Theorem~\ref{theorem-localization}), $\left(
S_{n,j}\left( f-\widetilde{f}\right) \right) \left( x\right) $ decays
exponentially fast when $n$ tends to infinity. In fact, we have the
following corollary.

\begin{corollary}
Let $j\in \mathbb{Z}$ and $x>0$. For $f\in E\cap C\left[ 0,+\infty \right) $%
, it holds 
\begin{equation*}
\lim_{n\rightarrow \infty }\left( S_{n,j}f\right) \left( x\right) =f\left(
x\right) .
\end{equation*}
\end{corollary}

In order to show Theorem~\ref{theorem-rate of convergence} we can apply the
classical estimate in terms of the second central moment (see \cite[Theorem
5.1.2]{Altomare}).

\begin{lemma}
\label{lemma-rate of convergence-Altomare-Campiti}If $Le_{0}=e_{0}$, then,
for $f\in C_{b}\left[ 0,+\infty \right) $, 
\begin{equation*}
\left\vert \left( Lf\right) \left( x\right) -f\left( x\right) \right\vert
\leq \left( 1+\frac{1}{\delta }\sqrt{\left( L\psi _{x}^{2}\right) \left(
x\right) }\right) \omega \left( f,\delta \right).
\end{equation*}
\end{lemma}

\begin{lemma}
\label{lemma-estimate-central-moment}Let $j\geq 2$. For $x\geq 0$ and
positive integers $n$, the first and the second central moment of $S_{n,j}$
satisfy the estimates 
\begin{equation*}
\frac{1-j}{n}\leq \left( S_{n,j}\psi _{x}^{1}\right) \left( x\right) \leq 0
\end{equation*}%
and 
\begin{equation*}
0\leq \left( S_{n,j}\psi _{x}^{2}\right) \left( x\right) \leq \frac{2x}{n}+%
\frac{\left( j-1\right) \left( j-2\right) }{n^{2}},
\end{equation*}%
respectively. For $j\leq 1$, we have 
\begin{equation*}
0\leq \left( S_{n,j}\psi _{x}^{1}\right) \left( x\right) =\frac{1-j}{n},%
\text{ \qquad }0\leq \left( S_{n,j}\psi _{x}^{2}\right) \left( x\right) =%
\frac{2x}{n}+\frac{\left( j-1\right) \left( j-2\right) }{n^{2}}.
\end{equation*}
\end{lemma}

\begin{proof}
The results for $j\leq 1$ follow immediately from Proposition \ref%
{prop-central-moments}. \newline
For $j\geq 2$, the first central moment is given by 
\begin{equation*}
\left( S_{n,j}\psi _{x}^{1}\right) \left( x\right) =\frac{1-j}{n}%
-e^{-nx}n^{-1}\sum_{k=0}^{j-2}\frac{\left( nx\right) ^{k}}{k!}\left(
k+1-j\right) =-\frac{j-1}{n}+\frac{1}{n}h_{1}\left( nx\right) ,
\end{equation*}%
where 
\begin{equation*}
h_{1}\left( t\right) =-e^{-t}\sum_{k=0}^{j-2}\frac{t^{k}}{k!}\left(
k+1-j\right) .
\end{equation*}%
Since $h_{1}\left( t\right) \geq 0$ and 
\begin{equation*}
h_{1}^{\prime }\left( t\right) =e^{-t}\sum_{k=0}^{j-2}\frac{t^{k}}{k!}\left(
k+1-j\right) -e^{-t}\sum_{k=0}^{j-2}\frac{t^{k}}{k!}\left( k+2-j\right)
=-e^{-t}\sum_{k=0}^{j-2}\frac{t^{k}}{k!}\leq 0
\end{equation*}%
we conclude that $0\leq h_{1}\left( t\right) \leq h_{1}\left( 0\right) =j-1$%
. This implies that $1-j\leq n\left( S_{n,j}\psi _{x}^{1}\right) \left(
x\right) \leq 0$. \newline
Now, let $j\geq 2$. The second central moment is given by 
\begin{equation*}
\left( S_{n,j}\psi _{x}^{2}\right) \left( x\right) =\frac{2x}{n}+\frac{%
\left( j-1\right) \left( j-2\right) }{n^{2}}+\frac{1}{n^{2}}h_{2}\left(
nx\right) ,
\end{equation*}%
where 
\begin{eqnarray*}
h_{2}\left( t\right) &=&e^{-t}\left[ 2t\sum_{k=0}^{j-2}\frac{t^{k}}{k!}%
\left( k+1-j\right) -\sum_{k=0}^{j-3}\frac{t^{k}}{k!}\left( k+2-j\right) ^{%
\underline{2}}\right] \\
&=&e^{-t}\sum_{k=0}^{j-1}\frac{t^{k}}{k!}\left[ 2k\left( k-j\right) -\left(
k+2-j\right) ^{\underline{2}}\right] \\
&=&e^{-t}\sum_{k=0}^{j-1}\frac{t^{k}}{k!}\left[ k(k-3)-\left( j-1\right)
\left( j-2\right) \right] .
\end{eqnarray*}%
Using the estimate 
\begin{equation*}
-2\leq k(k-3) \leq \left( j-4\right) \left( j-1\right) \text{ \qquad }\left(
k=1,\ldots ,j-1\right)
\end{equation*}%
we have 
\begin{eqnarray*}
\lefteqn{-\left( j-1\right) \left( j-2\right) e^{-t}+\left[ -2-\left(
j-1\right) \left( j-2\right) \right] e^{-t}\sum_{k=1}^{j-1}\frac{t^{k}}{k!}}
\\
&\leq& h_{2}\left( t\right) \leq -\left( j-1\right) e^{-t}\left(
j-2+2\sum_{k=1}^{j-1}\frac{t^{k}}{k!}\right) \leq 0.
\end{eqnarray*}%
Since $\left( S_{n,j}\psi _{x}^{2}\right) \left( x\right) \geq 0$, the
desired estimate of the second central moment follows.
\end{proof}

\begin{remark}
Note that $\left( S_{n,j}\psi _{x}^{2}\right) \left( x\right) =\dfrac{2x}{n}%
+O\left( n^{-2}\right) $ as $n\rightarrow \infty $.
\end{remark}

\begin{proof}[Proof of Theorem~\protect\ref{theorem-rate of convergence}]
By Lemma~\ref{lemma-rate of convergence-Altomare-Campiti} and Lemma~\ref%
{lemma-estimate-central-moment}, for $f\in C_{b}\left[ 0,+\infty \right) $
and $\delta >0$, 
\begin{equation*}
\left\vert \left( S_{n,j}f\right) \left( x\right) -f\left( x\right)
\right\vert \leq \left( 1+\frac{1}{\delta }\sqrt{\frac{2x}{n}+\frac{\left(
j-1\right) \left( j-2\right) }{n^{2}}}\right) \omega \left( f,\delta \right)
.
\end{equation*}%
Choosing $\delta =1/\sqrt{n}$ we get (\ref{estimate-rate of convergence}).
\end{proof}

\section{Localization result}

The purpose of this section is to estimate the rate of convergence of the
sequence of functions $S_{n,j}f$ at a point $x>0$ if $f\in E_{A}$ has the
specific property to vanish in a neighborhood of $x$. Though we need such an
estimate in deriving the asymptotic expansions for the sequence $\left(
S_{n,j}f\right) $, this so-called localization result is interesting in
itself.

\begin{theorem}
\label{theorem-localization}Let $f\in E$, $x\geq 0$ and $\delta >0$. If $%
f\left( t\right) =0$, for $t\in \left( x-\delta ,x+\delta \right) \cap \left[
0,\infty \right) $, then there is a positive constant $c$, such that $\left(
S_{n,j}f\right) \left( x\right) =O\left( e^{-cn}\right) $ as $n\rightarrow
\infty $.
\end{theorem}

A similar result for a general family of operators can be found in \cite[%
Theorem~8]{Lopez-Latorre-JMAA-2011}. For locally smooth functions $f$, the
conclusion is the weaker relation $\left( S_{n,j}f\right) \left( x\right)
=O\left( n^{-q}\right) $ as $n\rightarrow \infty $, for any positive
constant $q$.

For the proof of Theorem~\ref{theorem-localization} we need several
auxiliary results. We start with the following technical lemma.

\begin{lemma}
\label{lemma-log-estimate}The logarithm function satisfies the estimates 
\begin{eqnarray*}
\log t &<&t-1-\frac{1}{2}\left( 1-t\right) ^{2}\text{ \qquad }\left(
0<t<1\right) , \\
\log t &<&t-1-\frac{1}{2}\left( t-1\right) ^{2}\left( 2-t\right) \text{
\qquad }\left( 1<t<2\right) .
\end{eqnarray*}
\end{lemma}

\begin{proof}
Let $g(t):=t-1-\dfrac{1}{2}(1-t)^{2}-\log t$, $0<t\leq 1$. Then $g(1)=0$ and 
$g^{\prime }(t)=-\dfrac{(t-1)^{2}}{t}<0$, $0<t<1$. This proves the first
inequality. \newline
Now let $h(t):=t-1-\dfrac{1}{2}(t-1)^{2}(2-t)-\log t$, $1\leq t\leq 2$. Then 
$h(1)=0$ and $h^{\prime }(t)=\dfrac{(t-1)^{2}(3t-2)}{2t}>0$, $1<t\leq 2$,
which proves the second inequality.
\end{proof}

We use the following standard inequality \cite[Theorem 137, Eq. (9.1.4)]%
{Hardy-Divergent-Series-book}.

\begin{lemma}
\label{lemma-Hardy}Given $x,\delta >0$, there exists a constant $c>0$ such
that 
\begin{equation*}
e^{-x}\sum_{\substack{ k\geq 0  \\ \left\vert k-\left\lfloor x\right\rfloor
\right\vert >\delta x}}\frac{x^{k}}{k!}=O\left( e^{-cx}\right) \qquad \left(
x\rightarrow \infty \right) ,
\end{equation*}%
where $c=\delta ^{2}/3$.
\end{lemma}

\begin{lemma}
\label{lemma-estimate-szasz-tail}Let $\left( b_{n}\right) $ be a sequence of
positive numbers with $b_{n}\rightarrow 1$ as $n\rightarrow \infty $. Given $%
x,\beta >0$, there exists a constant $c>0$ such that 
\begin{equation*}
\sum_{\substack{ k\geq 0  \\ \left\vert k/n-x\right\vert >\beta }}%
s_{n,k}\left( x\right) b_{n}^{k}=O\left( e^{-cn}\right) \qquad \left(
n\rightarrow \infty \right) .
\end{equation*}
\end{lemma}

\begin{proof}
Let $b>1$ such that $b-1<\beta /\left( 2x\right) $ and $b<12/\left( 12-\beta
^{2}\right) $. Put $\gamma :=\beta /2$. Then $\left\vert k-nx\right\vert
>\beta n$ implies $\left\vert k-\left\lfloor nxb\right\rfloor \right\vert
>\gamma n$, for sufficiently large $n$. This can be seen as follows.
Firstly, we have $n\left( x-\beta \right) <n\left( x-\beta \right)
+\left\lfloor nxb\right\rfloor -nxb+1=\left\lfloor nxb\right\rfloor -n\gamma
+nx-nxb+1-n\gamma <\left\lfloor nxb\right\rfloor -n\gamma $, for $n>1/\gamma 
$. Secondly, we have $n\left( x+\beta \right) >n\left( x+\beta \right)
+\left\lfloor nxb\right\rfloor -nxb=\left\lfloor nxb\right\rfloor +n\gamma
+n\gamma -nx\left( b-1\right) >\left\lfloor nxb\right\rfloor +n\gamma
+n\gamma -nx\beta /\left( 2x\right) =\left\lfloor nxb\right\rfloor +n\gamma $%
. For sufficiently large $n$, we have $b_{n}<b$ and 
\begin{eqnarray*}
\sum_{\substack{ k\geq 0  \\ \left\vert k/n-x\right\vert >\beta }}%
s_{n,k}\left( x\right) b_{n}^{k} &=&e^{-nx}\sum_{\substack{ k\geq 0  \\ %
\left\vert k-\left\lfloor nxb\right\rfloor \right\vert >\gamma n}}\frac{%
\left( nxb_{n}\right) ^{k}}{k!}\leq e^{-nx}\sum_{\substack{ k\geq 0  \\ %
\left\vert k-\left\lfloor nxb\right\rfloor \right\vert >\gamma n}}\frac{%
\left( nxb\right) ^{k}}{k!} \\
&=&e^{\left( b-1\right) nx}O\left( e^{-\left( \gamma ^{2}/3\right)
nxb}\right) \qquad \left( n\rightarrow \infty \right) ,
\end{eqnarray*}%
by Lemma~\ref{lemma-Hardy}. Noting that $\left( b-1\right) -b\cdot \gamma
^{2}/3=b\left( 1-\gamma ^{2}/3\right) -1<12/\left( 12-\beta ^{2}\right)
\cdot \left( 1-\beta ^{2}/12\right) -1<0$ completes the proof.
\end{proof}

\begin{proof}[Proof of Theorem~\protect\ref{theorem-localization}]
Let $f\in E_{A}$. First consider the special case $x=0$. Since $\left(
S_{n,j}f\right) \left( 0\right) =0$, for $j\geq 1$, we assume that $j\leq 0$%
. If $f\left( t\right) =0$, for $t\in \left[ 0,x+\delta \right) $, then 
\begin{eqnarray*}
\left\vert \left( S_{n,j}f\right) \left( 0\right) \right\vert &\leq
&\sum_{k=0}^{\infty }s_{n,k}(0)n\int_{\delta }^{\infty
}s_{n,k-j}(t)\left\vert f(t)\right\vert dt\leq \frac{K}{\left( -j\right) !}%
n\int_{\delta }^{\infty }\left( nt\right) ^{-j}e^{-\left( n-A\right) t}dt \\
&=&\frac{K}{\left( -j\right) !}\int_{n\delta }^{\infty
}t^{-j}e^{-t}dt=O\left( e^{-cn}\right) \qquad \left( n\rightarrow \infty
\right) ,
\end{eqnarray*}%
for each positive $c<\delta $.\newline
Now let us turn to the case $x>0$. Let $f\left( t\right) =0$, if $x-\delta
<t<x+\delta $. Without loss of generality, we assume $\delta <x/3$.
Furthermore, put $\beta =\delta /2$. We have 
\begin{align}
\left( S_{n,j}f\right) \left( x\right) & =f\left( 0\right)
\sum_{k=0}^{j-1}s_{n,k}\left( x\right)  \notag \\
& +\left( \sum_{\substack{ k\geq 0  \\ \left\vert k/n-x\right\vert >\beta }}%
+\sum_{\substack{ k\geq 0  \\ \left\vert k/n-x\right\vert \leq \beta }}%
\right) s_{n,k+j}\left( x\right) \cdot n\int_{0}^{\infty }s_{n,k}\left(
t\right) f\left( t\right) dt  \notag \\
& =T_{0}+T_{1}+T_{2},  \label{W_n,j-composition}
\end{align}%
say. Obviously, 
\begin{equation*}
T_{0}=f\left( 0\right) \sum_{k=0}^{j-1}s_{n,k}\left( x\right) =O\left(
n^{j-1}e^{-nx}\right) \text{ \qquad }\left( n\rightarrow \infty \right) .
\end{equation*}%
Secondly, we estimate $\left\vert T_{1}\right\vert $. Since $\left\vert
f\left( t\right) \right\vert \leq Ke^{At}$, for $t\geq 0$, we have, for $n>A$%
, 
\begin{align*}
\left\vert n\int_{0}^{\infty }s_{n,k}\left( t\right) f\left( t\right)
dt\right\vert & \leq Kn\frac{1}{k!}\int_{0}^{\infty }e^{-nt}\left( nt\right)
^{k}e^{At}dt=K\left( 1+\frac{A}{n-A}\right) ^{k+1} \\
& \leq K\left( 1+\frac{A}{n-A}\right) e^{kA/\left( n-A\right) }.
\end{align*}%
Hence, for sufficiently large $n>2A$, there is a constant $c>0$, such that 
\begin{align*}
\left\vert T_{1}\right\vert & \leq 2K\sum_{\substack{ k\geq 0  \\ \left\vert
k/n-x\right\vert >\beta }}s_{n,k+j}\left( x\right) e^{kA/\left( n-A\right) }
\\
& =2K\sum_{\substack{ k\geq 0  \\ \left\vert k/n-x\right\vert >\beta }}%
s_{n,k}\left( x\right) \frac{k!}{\left( k+j\right) !}\left( nx\right)
^{j}e^{kA/\left( n-A\right) } \\
& \leq 2K\left( nx\right) ^{j}\sum_{\substack{ k\geq 0  \\ \left\vert
k/n-x\right\vert >\beta }}s_{n,k}\left( x\right) e^{2Ak/n}=O\left(
e^{-cn}\right) \text{ \qquad }\left( n\rightarrow \infty \right) ,
\end{align*}%
by an application of Lemma~\ref{lemma-estimate-szasz-tail} with $%
b_{n}=e^{2A/n}$. Now, we are going to estimate $\left\vert T_{2}\right\vert $
in $\left( \ref{W_n,j-composition}\right) $. We consider values $k=n\left(
x+\vartheta \right) $ with $\left\vert \vartheta \right\vert \leq \beta $.
Then, for $n>2A$, 
\begin{equation*}
U_{n,k}:=\left\vert n\int_{0}^{x-\delta }s_{n,k}\left( t\right) f\left(
t\right) dt\right\vert \leq K\frac{1}{k!}\left( \frac{n}{n-A}\right)
^{k+1}\int_{0}^{\left( x-\delta \right) \left( n-A\right) }e^{-t}t^{k}dt
\end{equation*}%
and 
\begin{align*}
V_{n,k}& :=\left\vert n\int_{x+\delta }^{\infty }s_{n,k}\left( t\right)
f\left( t\right) dt\right\vert \\
& \leq K\frac{1}{k!}\left( \frac{n}{n-2A}\right) ^{k+1}\int_{\left( x+\delta
\right) \left( n-2A\right) }^{\infty }e^{-t}t^{k}e^{-At/\left( n-2A\right)
}dt.
\end{align*}%
Define $g_{k}\left( t\right) =e^{-t}t^{k}$. Since $g_{k}^{\prime }\left(
t\right) =\left( k-t\right) t^{k-1}e^{-t}$, the function $g_{k}$ is unimodal
on $\left( 0,\infty \right) $ with a unique maximum at $t=k$. For
sufficiently large $n$, we have 
\begin{equation*}
U_{n,k}\leq K\frac{n}{k!}\left( \frac{n}{n-A}\right) ^{k}\left( x-\delta
\right) g_{k}\left( \left( x-\delta \right) \left( n-A\right) \right)
\end{equation*}%
and 
\begin{equation*}
V_{n,k}\leq \frac{K}{A}\frac{n}{k!}\left( \frac{n}{n-2A}\right)
^{k}g_{k}\left( \left( x+\delta \right) \left( n-2A\right) \right) ,
\end{equation*}%
since 
\begin{equation*}
\int_{\left( x+\delta \right) \left( n-2A\right) }^{\infty }e^{-At/\left(
n-2A\right) }dt=\frac{n-2A}{A}e^{-A\left( x+\delta \right) }\leq \frac{n-2A}{%
A}.
\end{equation*}%
Using Stirling's formula 
\begin{equation*}
k!=\sqrt{2\pi }k^{k+1/2}\exp \left( -k+\frac{\sigma }{12k}\right) ,\text{
\qquad }\left( k>0,\text{ }0<\sigma <1\right)
\end{equation*}%
for $k=n\left( x+\vartheta \right) $, we obtain 
\begin{eqnarray*}
U_{n,k} &\leq &\frac{K\left( x-\delta \right) ne^{k}}{\sqrt{2\pi }k^{k+1/2}}%
\left( \frac{n}{n-A}\right) ^{k}g_{k}\left( \left( x-\delta \right) \left(
n-A\right) \right) \\
&=&\frac{K\left( x-\delta \right) ne^{n\left( x+\vartheta \right) }}{\sqrt{%
2\pi }\left( n\left( x+\vartheta \right) \right) ^{n\left( x+\vartheta
\right) +1/2}}\left( \frac{n}{n-A}\right) ^{n\left( x+\vartheta \right) } \\
&&\times e^{-\left( x-\delta \right) \left( n-A\right) }\left( \left(
x-\delta \right) \left( n-A\right) \right) ^{n\left( x+\vartheta \right) } \\
&=&\frac{K\left( x-\delta \right) e^{A\left( x-\delta \right) }}{\sqrt{2\pi
\left( x+\vartheta \right) }}\sqrt{n}e^{n\left( \delta +\vartheta \right)
}\left( \frac{x-\delta }{x+\vartheta }\right) ^{n\left( x+\vartheta \right) }
\end{eqnarray*}%
and 
\begin{eqnarray*}
V_{n,k} &\leq &\frac{K}{A}\frac{ne^{k}}{\sqrt{2\pi }k^{k+1/2}}\left( \frac{n%
}{n-2A}\right) ^{k}g_{k}\left( \left( x+\delta \right) \left( n-2A\right)
\right) \\
&=&\frac{K}{A}\frac{ne^{n\left( x+\vartheta \right) }}{\sqrt{2\pi }\left(
n\left( x+\vartheta \right) \right) ^{n\left( x+\vartheta \right) +1/2}}%
\left( \frac{n}{n-2A}\right) ^{n\left( x+\vartheta \right) } \\
&&\times e^{-\left( x+\delta \right) \left( n-2A\right) }\left( \left(
x+\delta \right) \left( n-2A\right) \right) ^{n\left( x+\vartheta \right) }
\\
&=&\frac{K}{A}\frac{e^{2A\left( x+\delta \right) }}{\sqrt{2\pi \left(
x+\vartheta \right) }}\sqrt{n}e^{-n\left( \delta -\vartheta \right) }\left( 
\frac{x+\delta }{x+\vartheta }\right) ^{n\left( x+\vartheta \right) }.
\end{eqnarray*}%
Rewriting and applying Lemma~\ref{lemma-log-estimate} we obtain, with $%
t=\left( x-\delta \right) /\left( x+\vartheta \right) $, 
\begin{align*}
e^{\delta +\vartheta }\left( \frac{x-\delta }{x+\vartheta }\right)
^{x+\vartheta }& =\exp \left( \delta +\vartheta +\left( x+\vartheta \right)
\log \frac{x-\delta }{x+\vartheta }\right) \\
& <\exp \left( \delta +\vartheta +\left( x+\vartheta \right) \frac{-\delta
-\vartheta }{x+\vartheta }-\frac{\left( x+\vartheta \right) }{2}\left( \frac{%
\delta +\vartheta }{x+\vartheta }\right) ^{2}\right) \\
& =\exp \left( -\frac{\left( \delta +\vartheta \right) ^{2}}{2\left(
x+\vartheta \right) }\right) \leq \exp \left( -\frac{\left( \delta -\beta
\right) ^{2}}{2\left( x+\beta \right) }\right) \\
& =\exp \left( -\frac{\delta ^{2}}{4\left( 2x+\delta \right) }\right)
=:C_{1}\left( x,\delta \right)
\end{align*}%
and, with $t=\left( x+\delta \right) /\left( x+\vartheta \right) $, 
\begin{align*}
& e^{-\delta +\vartheta }\left( \frac{x+\delta }{x+\vartheta }\right)
^{x+\vartheta }=\exp \left( -\delta +\vartheta +\left( x+\vartheta \right)
\log \frac{x+\delta }{x+\vartheta }\right) \\
& <\exp \left( -\delta +\vartheta +\left( x+\vartheta \right) \frac{\delta
-\vartheta }{x+\vartheta }-\frac{\left( x+\vartheta \right) }{2}\left( \frac{%
\delta -\vartheta }{x+\vartheta }\right) ^{2}\left( 2-\frac{x+\delta }{%
x+\vartheta }\right) \right) \\
& =\exp \left( -\frac{\left( \delta -\vartheta \right) ^{2}}{2\left(
x+\vartheta \right) }\frac{x+2\vartheta -\delta }{x+\vartheta }\right) \\
& \leq \exp \left( -\frac{\left( \delta -\beta \right) ^{2}\left( x-2\beta
-\delta \right) }{2\left( x+\beta \right) ^{2}}\right) =\exp \left( -\frac{%
\delta ^{2}\left( x-2\delta \right) }{2\left( 2x+\delta \right) ^{2}}\right)
=:C_{2}\left( x,\delta \right) .
\end{align*}%
Without loss of generality, one can choose $\delta >0$ so small that both
constants are less than $1$. Now we can estimate $T_{2}$ in $\left( \ref%
{W_n,j-composition}\right) $. With $q:=\max \left\{ C_{1}\left( x,\delta
\right) ,C_{2}\left( x,\delta \right) \right\} $ we have 
\begin{align*}
\left\vert T_{2}\right\vert & \leq \sum_{\substack{ k\geq 0  \\ \left\vert
k/n-x\right\vert \leq \beta }}s_{n,k+j}\left( x\right) \cdot \left\vert
n\int_{0}^{\infty }s_{n,k}\left( t\right) f\left( t\right) dt\right\vert
\leq \sum_{k\geq 0}s_{n,k+j}\left( x\right) \cdot \left(
U_{n,k}+V_{n,k}\right) \\
& =\sqrt{n}O\left( q^{n}\right) \text{ \qquad }\left( n\rightarrow \infty
\right) .
\end{align*}%
Observing that $q<1$ concludes the proof.
\end{proof}

\section{Asymptotic expansions \label{section-Asymptotic expansions}}

In this section we derive the complete asymptotic expansion for the sequence
of operators $S_{n,j}$ in the form 
\begin{equation}
\left( S_{n,j}f\right) \left( x\right) \sim f(x)+\sum_{k=1}^{\infty
}c_{k,j}\left( f,x\right) \;n^{-k}\qquad (n\rightarrow \infty ),
\label{complete-asymptotic-expansion}
\end{equation}%
provided that $f$ admits derivatives of sufficiently high order at $x>0$.
Formula $\left( \ref{complete-asymptotic-expansion}\right) $ means that, for
all $q=0,1,2,\ldots $, it holds 
\begin{equation*}
\left( S_{n,j}f\right) \left( x\right) =\sum_{k=0}^{q}c_{k,j}\left(
f,x\right) \text{ }n^{-k}+o(n^{-q})\qquad (n\rightarrow \infty )
\end{equation*}%
where $c_{0,j}\left( f,x\right) =f\left( x\right) $. The coefficients $%
c_{k,j}\left( f,x\right) $, which are independent of $n$, will be given in a
concise form using the differential operator $\mathcal{D}_{j}$ as defined in
Eq. (\ref{eq.X2}). To this end, recall definition (\ref{eq.X3}), i.e., 
\begin{equation*}
\mathcal{D}_{j}^{2k}=e_{j}D^{k}(e_{k-j}D^{k}),\,\text{\qquad }k\in {\mathbb{N%
}},
\end{equation*}%
and relation (\ref{eq.X1}), i.e., $\mathcal{D}_{j}^{2k}=\left( \mathcal{D}%
_{j}^{2}\right) ^{k}$.

\begin{theorem}
\label{theorem-asymptotic expansion}Let $q$ be a positive integer. For each
function $f\in E$ having the derivative $f^{\left( 2q\right) }\left(
x\right) $, the operators $S_{n,j}$ possess the asymptotic expansion 
\begin{equation*}
\left( S_{n,j}f\right) \left( x\right) =f\left( x\right) +\sum_{k=1}^{q}%
\frac{1}{k!n^{k}}\left( {\mathcal{D}}_{j}^{2k}f\right) \left( x\right)
+o\left( n^{-q}\right) \text{ \qquad }\left( n\rightarrow \infty \right) .
\end{equation*}
\end{theorem}

\begin{remark}
For a function $f\in E$ having all derivatives at $x>0$, we have the
complete asymptotic expansion 
\begin{equation*}
\left( S_{n,j}f\right) \left( x\right) \sim f\left( x\right)
+\sum_{k=1}^{\infty }\frac{1}{k!n^{k}}\left( {\mathcal{D}}_{j}^{2k}f\right)
\left( x\right) \text{ \qquad }\left( n\rightarrow \infty \right) .
\end{equation*}
\end{remark}

\begin{remark}
Note that $\left( {\mathcal{D}}_{j}^{2k}e_{j}\right) \left( x\right) =0$,
for all $k\in \mathbb{N}$. This reflects the fact that $S_{n,j}$ preserves
the monomial $e_{j}$.
\end{remark}

\begin{remark}
In the special case $j=0$, Theorem~\ref{theorem-asymptotic expansion}
recovers the complete asymptotic expansion 
\begin{equation*}
\left( S_{n,0}f\right) \left( x\right) \sim f\left( x\right)
+\sum_{k=1}^{\infty }\frac{1}{k!n^{k}}\left( x^{k}f^{\left( k\right) }\left(
x\right) \right) ^{\left( k\right) }\text{ \qquad }\left( n\rightarrow
\infty \right)
\end{equation*}%
for the classical Sz\'{a}sz--Mirakjan--Durrmeyer operators. This result can
be found in \cite[Corollary~2.4 with $r=0$]{Abel-BDV-Ivan-Maratea-2004}.\ In
the special case $j=1$ we get the relation 
\begin{equation*}
\left( S_{n,1}f\right) \left( x\right) \sim f\left( x\right)
+\sum_{k=1}^{\infty }\frac{x}{k!n^{k}}\left( x^{k-1}f^{\left( k\right)
}\left( x\right) \right) ^{\left( k\right) }\text{ \qquad }\left(
n\rightarrow \infty \right)
\end{equation*}%
for the genuine Sz\'{a}sz--Mirakjan--Durrmeyer operators.
\end{remark}

In order to derive Theorem~\ref{theorem-asymptotic expansion}, a general
approximation theorem due to Sikkema \cite[Theorem~3]{Sikkema-1970a} will be
applied. For $k\in \mathbb{N}$ and $x>0$, let $H^{\left( k\right) }\left(
x\right) $ denote the class of all locally bounded real functions $f:\left[
0,\infty \right) \rightarrow \mathbb{R}$, which are $k$ times differentiable
at $x$, and satisfy the additional condition $f\left( t\right) =O\left(
t^{-k}\right) $ as $t\rightarrow +\infty $. An inspection of the proof of
Sikkema's result reveals that it can be stated in the following form which
is more appropriate for our purposes.

\begin{lemma}
\label{Lemma-Sikkema}Let $q\in \mathbb{N}$ and let $\left( L_{n}\right)
_{n\in \mathbb{N}}$ be a sequence of positive linear operators, $%
L_{n}:H^{\left( 2q\right) }\left( x\right) \rightarrow C\left[ 0,+\infty
\right) $, $x\in \left[ 0,+\infty \right) $. Suppose that the operators $%
L_{n}$ apply to $\psi _{x}^{2q+1}$ and to $\psi _{x}^{2q+2}$. Then the
condition 
\begin{equation*}
\left( L_{n}\psi _{x}^{s}\right) \left( x\right) =O\left( n^{-\left\lfloor
\left( s+1\right) /2\right\rfloor }\right) \text{ \qquad }\left(
n\rightarrow \infty \right) ,\text{ \qquad for }s=0,1,\ldots ,2q+2,
\end{equation*}%
implies, for each function $f\in H^{\left( 2q\right) }\left( x\right) $, the
asymptotic relation 
\begin{equation*}
\left( L_{n}f\right) \left( x\right) =\sum_{s=0}^{2q}\frac{f^{\left(
s\right) }\left( x\right) }{s!}\left( L_{n}\psi _{x}^{s}\right) \left(
x\right) +o\left( n^{-q}\right) \text{ }\qquad \left( n\rightarrow \infty
\right) .
\end{equation*}
\end{lemma}

\begin{proof}[Proof of Theorem~\protect\ref{theorem-asymptotic expansion}]
Let $f\in E$. Given $x>0$, put $U_{r}\left( x\right) =\left( x-r,x+r\right)
\cap \left[ 0,+\infty \right) $, for $r>0$. Let $\delta >0$ be given.
Suppose that $f^{\left( 2q\right) }\left( x\right) $ exists. Choose a
function $\varphi \in C^{\infty }\left( \left[ 0,+\infty \right) \right) $
with $\varphi \left( x\right) =1$ on $U_{\delta }\left( x\right) $ and $%
\varphi \left( x\right) =0$ on $\left[ 0,+\infty \right) \setminus
U_{2\delta }\left( x\right) $. Put $\widetilde{f}=\varphi f$. Then we have $%
\widetilde{f}\equiv f$ on $U_{\delta }\left( x\right) $ which implies $%
\widetilde{f}^{\left( j\right) }\left( x\right) =f^{\left( j\right) }\left(
x\right) $, for $j=0,\ldots ,2q$, and $\widetilde{f}\equiv 0$ on $\left[
0,+\infty \right) \setminus U_{2\delta }\left( x\right) $. By the
localization theorem (Theorem~\ref{theorem-localization}), $\left(
S_{n,j}\left( f-\widetilde{f}\right) \right) \left( x\right) $ decays
exponentially fast as $n\rightarrow \infty $. Consequently, $\widetilde{f}$
and $f$ possess the same asymptotic expansion of the form $\left(
S_{n,j}f\right) \left( x\right) =\sum_{k=0}^{q}a_{k,j}\left( f,x\right)
n^{-k}+o\left( n^{-q}\right) $ as $n\rightarrow \infty $. Therefore, without
loss of generality, we can assume that $f\equiv 0$ on $\left[ 0,+\infty
\right) \setminus U_{2\delta }\left( x\right) $. By Proposition~\ref%
{prop-central-moments}, we have $\left( S_{n,j}\psi _{x}^{2q}\right) \left(
x\right) =O\left( n^{-q}\right) $ as $n\rightarrow \infty $. Under these
conditions, Lemma~\ref{Lemma-Sikkema} implies that 
\begin{equation*}
\left( S_{n,j}f\right) \left( x\right) =\sum_{s=0}^{2q}\frac{f^{\left(
s\right) }\left( x\right) }{s!}\left( S_{n,j}\psi _{x}^{s}\right) \left(
x\right) +o\left( n^{-q}\right) \text{ }\qquad \left( n\rightarrow \infty
\right) .
\end{equation*}%
By Proposition~\ref{prop-central-moments}, we obtain 
\begin{eqnarray*}
\left( S_{n,j}f\right) \left( x\right) &=&\sum_{s=0}^{2q}\frac{f^{\left(
s\right) }\left( x\right) }{s!}s!\sum_{k=\left\lfloor \left( s+1\right)
/2\right\rfloor }^{s}\frac{x^{s-k}}{\left( s-k\right) !n^{k}}\binom{k-j}{2k-s%
}+o\left( n^{-q}\right) \\
&=&\sum_{k=0}^{q}\frac{1}{n^{k}}\sum_{s=k}^{2k}\frac{f^{\left( s\right)
}\left( x\right) }{\left( s-k\right) !}\binom{k-j}{2k-s}x^{s-k}+o\left(
n^{-q}\right) \\
&=&\sum_{k=0}^{q}\frac{1}{n^{k}}\sum_{s=0}^{k}\frac{f^{\left( k+s\right)
}\left( x\right) }{s!}\binom{k-j}{k-s}x^{s}+o\left( n^{-q}\right) \text{ }%
\qquad \left( n\rightarrow \infty \right) .
\end{eqnarray*}%
Interchanging the order of summation, we obtain 
\begin{equation*}
\left( S_{n,j}f\right) \left( x\right) =\sum_{k=0}^{q}\frac{1}{n^{k}}%
\sum_{s=k}^{2k}\frac{f^{\left( s\right) }\left( x\right) }{\left( s-k\right)
!}\binom{k-j}{2k-s}x^{s-k}+o\left( n^{-q}\right) \text{ }\qquad \left(
n\rightarrow \infty \right) .
\end{equation*}%
Observing that 
\begin{align*}
&\sum_{s=k}^{2k}\frac{f^{\left( s\right) }\left( x\right) }{\left(
s-k\right) !}\binom{k-j}{2k-s}x^{s-k} =\sum_{s=0}^{k}\frac{f^{\left(
k+s\right) }\left( x\right) }{s!}\binom{k-j}{k-s}x^{s} \\
&=\frac{x^{j}}{k!}\sum_{s=0}^{k}\binom{k}{s}f^{\left( k+s\right) }\left(
x\right) \left( x^{k-j}\right) ^{\left( k-s\right) }=\frac{x^{j}}{k!}\left(
x^{k-j}f^{\left( k\right) }\left( x\right) \right) ^{\left( k\right) },
\end{align*}%
where the last equation follows from Leibniz rule for differentiation, we
obtain 
\begin{equation*}
\left( S_{n,j}f\right) \left( x\right) =x^{j}\sum_{k=0}^{q}\frac{1}{k!n^{k}}%
\left( x^{k-j}f^{\left( k\right) }\left( x\right) \right) ^{\left( k\right)
}+o\left( n^{-q}\right) \text{ }\qquad \left( n\rightarrow \infty \right) .
\end{equation*}%
Using the notation from Eq. $\left( \ref{eq.X3}\right) $ we get the result
in the desired form, which completes the proof.
\end{proof}

Now, we study simultaneous approximation by the operators $S_{n,j}$. It
turns out that $\left( S_{n,j}f\right) ^{\left( m\right) }\left( x\right) $
possess as asymptotic expansion which arises from the asymptotic expansion
for $\left( S_{n,j}f\right) \left( x\right) $ by differentiating its
coefficient term-by-term with respect to the variable $x$.

\begin{theorem}
Let $x>0$ and $q,m$ non-negative integers. If the function $f\in E$ admits
the derivative $f^{\left( 2q+2m\right) }\left( x\right) $, then the
derivatives of $S_{n,j}f$ possess the asymptotic expansion 
\begin{equation}
\left( S_{n,j}f\right) ^{\left( m\right) }\left( x\right)
=\sum_{k=0}^{q}c_{k,j}^{\left( m\right) }\left( f,x\right) n^{-k}+o\left(
n^{-q}\right) \text{ \qquad }\left( n\rightarrow \infty \right) ,
\label{Asymptotic-expansion-simultaneous}
\end{equation}%
where 
\begin{equation*}
c_{k,j}\left( f,x\right) =\frac{1}{k!}\left( {\mathcal{D}}_{j}^{2k}f\right)
\left( x\right) =\frac{1}{k!}x^{j}\left( x^{k-j}f^{\left( k\right) }\left(
x\right) \right) ^{\left( k\right) }.
\end{equation*}
In a more explicit form it holds 
\begin{equation}
\left( S_{n,j}f\right) ^{\left( m\right) }\left( x\right) =\sum_{k=0}^{q}%
\frac{1}{n^{k}}\sum_{i=0}^{k}\frac{1}{i!}\binom{k+m-j}{k-i}x^{i}f^{\left(
k+m+i\right) }\left( x\right) +o\left( n^{-q}\right)
\label{Asymptotic-expansion-simultaneous-explicit}
\end{equation}%
as $n\rightarrow \infty $.
\end{theorem}

For the classical Sz\'{a}sz--Mirakjan--Durrmeyer operators, i.e., for $j=0$,
this result can be found in \cite[Corollary~2.4]{Abel-BDV-Ivan-Maratea-2004}%
.\ 

\begin{proof}
Let $f\in E$. Using 
\begin{equation*}
s_{n,k}^{\left( m\right) }\left( x\right) =\left( \frac{d}{dx}\right)
^{m}\left( e^{-nx}\frac{\left( nx\right) ^{k}}{k!}\right) =n^{m}\sum_{\ell
=0}^{m}\left( -1\right) ^{m-\ell }\binom{m}{\ell }s_{n,k-\ell }\left(
x\right)
\end{equation*}%
it follows that 
\begin{equation*}
\left( S_{n,j}f\right) ^{\left( m\right) }\left( x\right) =n^{m}\sum_{\ell
=0}^{m}\left( -1\right) ^{m-\ell }\binom{m}{\ell }\left( S_{n,j-\ell
}f\right) \left( x\right) .
\end{equation*}%
If $f^{\left( 2q+2m\right) }\left( x\right) $ exists, Theorem~\ref%
{theorem-asymptotic expansion} yields the asymptotic expansion 
\begin{equation*}
\left( S_{n,j}f\right) \left( x\right) =f\left( x\right) +\sum_{k=1}^{q+m}%
\frac{1}{k!n^{k}}\left( {\mathcal{D}}_{j}^{2k}f\right) \left( x\right)
+o\left( n^{-q-m}\right) \text{ \qquad }\left( n\rightarrow \infty \right) .
\end{equation*}%
Hence, we obtain 
\begin{align*}
\left( S_{n,j}f\right) ^{\left( m\right) }\left( x\right) &=n^{m}\sum_{\ell
=0}^{m}\left( -1\right) ^{m-\ell }\binom{m}{\ell }\left( \sum_{k=0}^{q+m}%
\frac{1}{k!n^{k}}\left( {\mathcal{D}}_{j-\ell }^{2k}f\right) \left( x\right)
+o\left( n^{-q-m}\right) \right) \\
&=\sum_{k=0}^{q+m}\frac{1}{k!n^{k-m}}\sum_{\ell =0}^{m}\left( -1\right)
^{m-\ell }\binom{m}{\ell }\left( {\mathcal{D}}_{j-\ell }^{2k}f\right) \left(
x\right) +o\left( n^{-q}\right) \,\,\left( n\rightarrow \infty \right) .
\end{align*}%
We have 
\begin{eqnarray*}
\sum_{\ell =0}^{m}\left( -1\right) ^{m-\ell }\binom{m}{\ell }\left( {%
\mathcal{D}}_{j-\ell }^{2k}f\right) \left( x\right) &=&\sum_{\ell
=0}^{m}\left( -1\right) ^{m-\ell }\binom{m}{\ell }x^{j-\ell }\left(
x^{k-j+\ell }f^{\left( k\right) }\left( x\right) \right) ^{\left( k\right) }
\\
&=&x^{j}\left( \left( \frac{x}{t}-1\right) ^{m}x^{k-j}f^{\left( k\right)
}\left( x\right) \right) ^{\left( k\right) }\mid _{t=x} \\
&=&x^{j}\binom{k}{m}m!x^{-m}\left( x^{k-j}f^{\left( k\right) }\left(
x\right) \right) ^{\left( k-m\right) },
\end{eqnarray*}%
which vanishes, for $0\leq k<m$. Hence, 
\begin{equation*}
\left( S_{n,j}f\right) ^{\left( m\right) }\left( x\right) =\sum_{k=0}^{q}%
\frac{1}{k!n^{k}}x^{j-m}\left( x^{k+m-j}f^{\left( k+m\right) }\left(
x\right) \right) ^{\left( k\right) }+o\left( n^{-q}\right) \,\,\left(
n\rightarrow \infty \right) .
\end{equation*}%
Application of the Leibniz rule yields 
\begin{equation*}
\left( S_{n,j}f\right) ^{\left( m\right) }\left( x\right) =\sum_{k=0}^{q}%
\frac{1}{k!n^{k}}\sum_{i=0}^{k}\binom{k}{i}\left( k+m-j\right) ^{\underline{%
k-i}}x^{i}f^{\left( k+m+i\right) }\left( x\right) +o\left( n^{-q}\right)
\,\,\left( n\rightarrow \infty \right) ,
\end{equation*}%
which proves $\left( \ref{Asymptotic-expansion-simultaneous-explicit}\right) 
$. Now we calculate 
\begin{align*}
k!c_{k,j}^{\left( m\right) }\left( f,x\right) &=\left( x^{j}\left(
x^{k-j}f^{\left( k\right) }\left( x\right) \right) ^{\left( k\right)
}\right) ^{\left( m\right) } \\
&=\sum_{\ell =0}^{m}\binom{m}{\ell }j^{\underline{m-\ell }}x^{j-m+\ell
}\left( x^{k-j}f^{\left( k\right) }\left( x\right) \right) ^{\left( k+\ell
\right) } \\
&=\sum_{\ell =0}^{m}\binom{m}{\ell }j^{\underline{m-\ell }}x^{j-m+\ell
}\sum_{i=0}^{k+\ell }\binom{k+\ell }{i}\left( k-j\right) ^{\underline{k+\ell
-i}}x^{-\ell +i-j}f^{\left( k+i\right) }\left( x\right) \\
&=\sum_{i=0}^{k+m}x^{i-m}f^{\left( k+i\right) }\left( x\right) d\left(
k,j,m,i\right) ,
\end{align*}%
where 
\begin{equation*}
d\left( k,j,m,i\right) =\sum_{\ell =0}^{m}\binom{m}{\ell }\binom{k+\ell }{i}%
j^{\underline{m-\ell }}\left( k-j\right) ^{\underline{k+\ell -i}}.
\end{equation*}%
Application of Vandermonde convolution 
\begin{equation*}
\binom{k+\ell }{i}=\sum_{p=0}^{i}\binom{k}{i-p}\binom{\ell }{p}
\end{equation*}%
and the binomial identity $\binom{m}{\ell }\binom{\ell }{p}=\binom{m}{p}%
\binom{m-p}{\ell -p}$ leads to 
\begin{eqnarray*}
d\left( k,j,m,i\right) &=&\sum_{p=0}^{i}\binom{k}{i-p}\binom{m}{p}\sum_{\ell
=p}^{m}\binom{m-p}{\ell -p}j^{\underline{m-\ell }}\left( k-j\right) ^{%
\underline{k+\ell -i}} \\
&=&\sum_{p=0}^{i}\binom{k}{i-p}\binom{m}{p}\sum_{\ell =0}^{m-p}\binom{m-p}{%
\ell }j^{\underline{m-p-\ell }}\left( k-j\right) ^{\underline{k+\ell +p-i}}.
\end{eqnarray*}%
A second application of Vandermonde convolution shows that the inner sum is
equal to 
\begin{equation*}
\left( k-j\right) ^{\underline{k+p-i}}\sum_{\ell =0}^{m-p}\binom{m-p}{\ell }%
j^{\underline{m-p-\ell }}\left( i-p-j\right) ^{\underline{\ell }}=\left(
k-j\right) ^{\underline{k+p-i}}\left( i-p\right) ^{\underline{m-p}}.
\end{equation*}%
By a third application of Vandermonde convolution 
\begin{equation*}
\left( k-j\right) ^{\underline{k+p-i}}=\sum_{\mu =0}^{k+p-i}\binom{k+p-i}{%
\mu }\left( k+m-j\right) ^{\underline{q}}\left( -m\right) ^{\underline{%
k+p-i-q}}
\end{equation*}%
and reversing the summation order of the first sum we obtain 
\begin{equation*}
d\left( k,j,m,i\right) =\sum_{p=0}^{i}\binom{k}{p}\binom{m}{i-p}p^{%
\underline{m-i+p}}\sum_{\mu =0}^{k-p}\binom{k-p}{\mu }\left( k+m-j\right) ^{%
\underline{q}}\left( -m\right) ^{\underline{k-p-q}}.
\end{equation*}%
Observing that $p^{\underline{m-i+p}}=0$ if $i<m$, we infer that 
\begin{equation}
k!c_{k,j}^{\left( m\right) }\left( f,x\right) =\sum_{i=0}^{k}x^{i}f^{\left(
k+m+i\right) }\left( x\right) d\left( k,j,m,i+m\right) ,
\label{intermediate-formula-for-koeff-derivatives}
\end{equation}%
with 
\begin{equation*}
d\left( k,j,m,i+m\right)=
\end{equation*}
\begin{equation*}
\sum_{\mu =0}^{k}\left( k+m-j\right) ^{\underline{\mu }}\sum_{p=0}^{i+m}%
\binom{k}{p}\binom{k-p}{\mu }\binom{m}{i+m-p}p^{\underline{p-i}}\left(
-m\right) ^{\underline{k-p-\mu }}.
\end{equation*}%
Since $\binom{m}{i+m-p}=0$, if $i>p$, and $\binom{m}{i+m-p}=\binom{m}{p-i}$,
for $i\leq p$, we obtain 
\begin{equation*}
d\left( k,j,m,i+m\right)=
\end{equation*}
\begin{equation*}
\sum_{\mu =0}^{k}\left( k+m-j\right) ^{\underline{\mu }}\sum_{p=0}^{m}\binom{%
k}{p+i}\binom{k-i-p}{\mu }\binom{m}{p}\left( p+i\right) ^{\underline{p}%
}\left( -m\right) ^{\underline{k-i-p-\mu }}.
\end{equation*}%
Using $\binom{k}{p+i}\left( p+i\right) ^{\underline{p}}i!=k^{\underline{p+i}%
}=k^{\underline{i}}\left( k-i\right) ^{\underline{p}}$ this equation can be
rewritten in the form 
\begin{equation*}
d\left( k,j,m,i+m\right)=
\end{equation*}
\begin{equation*}
\sum_{\mu =0}^{k}\left( k+m-j\right) ^{\underline{\mu }}\binom{k}{i}%
\sum_{p=0}^{k-i}\binom{k-i}{p}\binom{k-i-p}{\mu }\binom{m}{p}\left(
-m\right) ^{\underline{k-i-p-\mu }}.  \label{formula-for-d(k,j,m,i+m)}
\end{equation*}%
Because $\binom{k-i-p}{\mu }=0$, if $\mu >k-i-p$, the inner sum of $\left( %
\ref{formula-for-d(k,j,m,i+m)}\right) $ vanishes if $k-i<\mu \leq k$. In the
case $0\leq \mu <k-i$, we use $p!\binom{m}{p}\left( -m\right) ^{\underline{%
k-i-p-\mu }}=\left( -1\right) ^{k-i-p-\mu }m\left( m+k-i-p-\mu -1\right) ^{%
\underline{k-i-\mu -1}}$ to see that the inner sum of $\left( \ref%
{formula-for-d(k,j,m,i+m)}\right) $ is equal to 
\begin{equation*}
\sum_{p=0}^{k-i}\binom{k-i}{p}\binom{k-i-p}{\mu }m\left( m+k-i-p-\mu
-1\right) ^{\underline{k-i-\mu -1}}.
\end{equation*}%
Since $\binom{k-i-p}{\mu }m\left( m+k-i-p-\mu -1\right) ^{\underline{k-i-\mu
-1}}$ is a polynomial of degree $k-i-1$ in the variable $p$, this sum
vanishes. Finally, if $\mu =k-i$, the inner sum of $\left( \ref%
{formula-for-d(k,j,m,i+m)}\right) $ has the value 
\begin{equation*}
\sum_{p=0}^{k-i}\binom{k-i}{p}\binom{k-i-p}{k-i}\binom{m}{p}\left( -m\right)
^{\underline{-p}}=1.
\end{equation*}%
Summarizing, we obtain 
\begin{equation*}
d\left( k,j,m,i+m\right) =\left( k+m-j\right) ^{\underline{k-i}}\binom{k}{i}.
\end{equation*}%
By $\left( \ref{intermediate-formula-for-koeff-derivatives}\right) $, we get 
\begin{equation*}
c_{k,j}^{\left( m\right) }\left( f,x\right) =\frac{1}{k!}\sum_{i=0}^{k}%
\binom{k}{i}\left( k+m-j\right) ^{\underline{k-i}}x^{i}f^{\left(
k+m+i\right) }\left( x\right) ,
\end{equation*}%
and comparison with $\left( \ref{Asymptotic-expansion-simultaneous-explicit}%
\right) $ proves $\left( \ref{Asymptotic-expansion-simultaneous}\right) $.
\end{proof}

$  $

\noindent\textbf{Acknowledgement.} The second author was supported by Lucian Blaga
University of Sibiu (Knowledge Transfer Center) $\&$ Hasso Plattner
Foundation research grants LBUS-HPI-ERG-2023-01.

\end{document}